\documentclass[12pt]{amsart}
\usepackage[dvips]{graphicx}
\usepackage{amssymb,amsmath,amsthm}
\def\squarefree{C_A^{\ {\rm sf}}}
\def\squarefreesquarefree{C_A^{\ {\rm sfsf}}}
\def\squarefreeG{C_{A_G}^{\ {\rm sf}}}
\def\squarefreesquarefreeG{C_{A_G}^{\ {\rm sfsf}}}

\def\bold{\bf}

\def\0b{{\bold 0}}
\usepackage{bm}
\bmdefine{\Bzero}{0}
\bmdefine{\Bone}{1}
\def\Bone{{\bf 1}}
\def\RR{{\mathbb R}}
\def\ZZ{{\mathbb Z}}

\newtheorem{Theorem}{Theorem}[section]
\newtheorem{Lemma}[Theorem]{Lemma}
\newtheorem{Corollary}[Theorem]{Corollary}
\newtheorem{Proposition}[Theorem]{Proposition}

\newtheorem{Example}[Theorem]{Example}

\begin{document}
\title{Toric ideals and their circuits}
\author{Hidefumi Ohsugi and Takayuki Hibi}
\thanks{This research is supported by JST CREST}
\address{Hidefumi Ohsugi,
Department of Mathematics,
College of Science,
Rikkyo University,
Toshima-ku, Tokyo 171-8501, Japan} 
\email{ohsugi@rikkyo.ac.jp}
\address{Takayuki Hibi,
Department of Pure and Applied Mathematics,
Graduate School of Information Science and Technology,
Osaka University,
Toyonaka, Osaka 560-0043, Japan}
\email{hibi@math.sci.osaka-u.ac.jp}
\date{}

\begin{abstract}
In this paper, we study toric ideals
generated by circuits.
For toric ideals which have squarefree quadratic initial ideals,
a sufficient condition to be generated by circuits is given.
In particular, squarefree Veronese subrings,
the second Veronese subrings and configurations
arising from root systems satisfy the condition.
In addition, we study toric ideals of finite graphs
and characterize the graphs 
whose toric ideals are generated by circuits
$u -v$ such that 
either $u$ or $v$ is squarefree.
There exists several classes of graphs whose toric ideals
satisfy this condition and whose toric rings are nonnormal.
\end{abstract}

\maketitle


\section*{Introduction}

Let $\ZZ^{d \times n}$ be the set of all $d \times n$ integer matrices.
A {\em configuration} of $\RR^d$ is a matrix $A \in \ZZ^{d \times n}$, 
for which
there exists a hyperplane ${\mathcal H} \subset \RR^d$ 
not passing the origin of $\RR^d$ such that each column vector 
of $A$ lies on ${\mathcal H}$.
Throughout this paper, we assume that
the columns of $A$ are pairwise distinct.
Let $K$ be a field and 
$K[T, T^{-1}] 
= K[t_1, t_1^{-1}, \ldots, t_d, t_d^{-1}]$
the Laurent polynomial ring in $d$ variables over $K$.
Each column vector 
${\bf a} = (a_1, \ldots, a_d)^\top \in \ZZ^d$ 
$( = \ZZ^{d \times 1})$, where
$(a_1, \ldots, a_d)^\top$ is the transpose of
$(a_1, \ldots, a_d)$,
yields the Laurent monomial $T^{\bf a} =
t_1^{a_1} \cdots t_d^{a_d}$.
Let $A \in \ZZ^{d \times n}$ be a configuration of $\RR^d$
with ${\bf a}_1, \ldots, {\bf a}_n$ its column vectors.   
The {\em toric ring} of $A$ 
is the subalgebra $K[A]$ of $K[T, T^{-1}]$ 
which is generated by the Laurent monomials
$T^{{\bf a}_1}, \ldots, T^{{\bf a}_n}$ over $K$.
Let $K[X] = K[x_1, \ldots, x_n]$ be the polynomial
ring in $n$ variables over $K$ and define the surjective
ring homomorphism $\pi : K[X] \rightarrow K[A]$
by setting $\pi(x_i) = T^{{\bf a}_i}$ 
for $i = 1, \ldots, n$. 
We say that the kernel $I_A \subset K[X]$ of $\pi$ is 
the {\em toric ideal} of $A$.
It is known that, if $I_A \neq \{0\}$, then
 $I_A$ is generated by homogeneous binomials of degree $\geq 2$.
More precisely,
$$
I_A
=
\left<
\left.
X^{{\bf u}^+} - \ X^{{\bf u}^-}  \in K[X] \ \right| \  
{\bf u} \in {\rm Ker}_\ZZ (A) 
\right>
$$
where
$
{\rm Ker}_\ZZ (A) = 
\{
{\bf u} \in \ZZ^n \ | \  A {\bf u} = {\bf 0}
\}
$.
Here ${\bf u}^+ \in \ZZ_{\geq 0}^n$ (resp.~${\bf u}^- \in \ZZ_{\geq 0}^n$) is the positive part (resp.~negative part) of ${\bf u} \in \ZZ^n$.
In particular, we have ${\bf u} = {\bf u}^+ - {\bf u}^-$.
See \cite{Stu} for details.

The support of a monomial $u$ of $K[X]$ is
${\rm supp} (u) = \{x_i \ | \ x_i \mbox{ divides } u \}$
and the support of a binomial $f=u-v$ is ${\rm supp}(f) = {\rm supp}(u) \cup {\rm supp}(v)$.
We say that an irreducible binomial $f \in I_A$ is a {\em circuit}
of $I_A$ if there is no binomial $g \in I_A$ such that
${\rm supp}(g) \subset {\rm supp}(f )$ and ${\rm supp}(g) \neq {\rm supp}(f )$.
Note that a binomial $f \in I_A$ is a circuit of $I_A$
if and only if $I_A \cap K[\{x_i \ | \ x_i \in {\rm supp}(f )\}]$ is
generated by $f$.
Let $C_A$ be the set of circuits of $I_A$
and define its subsets 
$\squarefree$ and $\squarefreesquarefree$
by
\begin{eqnarray*}
\squarefree
&=&
\{
X^{\bf u} - X^{\bf v} \in C_A \ | \ 
\mbox{ either } X^{\bf u} \mbox{ or } X^{\bf v}
\mbox{ is squarefree }
\},\\
\squarefreesquarefree
&=&
\{
X^{\bf u} - X^{\bf v} \in C_A \ | \ 
\mbox{ both } X^{\bf u} \mbox{ and } X^{\bf v}
\mbox{ are squarefree }
\}.
\end{eqnarray*}
It is known \cite[Proposition 4.11]{Stu} that $C_A \subset {\mathcal U}_A$ where
${\mathcal U}_A$ is the union of all reduced Gr\"obner bases of $I_A$.
Since any Gr\"obner basis is a set of generators,
we have $I_A = \langle {\mathcal U}_A \rangle$.
Bogart--Jensen--Thomas \cite{Rekha} characterized
the configuration $A$ such that $I_A = \langle C_A \rangle$
in terms of polytopes.
On the other hand,
Martinez-Bernal--Villarreal \cite{V}
introduced the notion of ``unbalanced circuits" and
characterized
the configuration $A$ such that $I_A = \langle C_A \rangle$
in terms of unbalanced circuits when $K[A]$ is normal.
Note that, if $K[A]$ is normal, then
any binomial belonging to
a minimal set of binomial generators of $I_A$ 
has a squarefree monomial.
(This fact appeared in many papers.
See, e.g., \cite[Lemma 6.1]{OHtable}.)

One of the most important classes of toric ideals
whose circuits are well-studied is toric ideals arising from
finite graphs.
Let $G$ be a finite connected graph on the vertex set $[d] = \{1,2,\ldots,d\}$
with the edge set $E(G) = \{e_1,\ldots, e_n\}$.
Let ${\bf e}_1, \ldots, {\bf e}_d$
stand for the canonical unit coordinate vector of $\RR^d$.
If $e = \{ i, j \}$ is an edge of $G$, then
the column vector $\rho(e) \in \RR^d$
is defined by $\rho(e) = {\bf e}_i + {\bf e}_j$.
Let $A_G \in \ZZ^{d \times n}$ denote the matrix 
with column vectors $\rho(e_1), \ldots, \rho(e_n)$.
Then $A_G$ is a configuration of $\RR^d$
which is the vertex-edge incidence matrix of $G$. 
Circuits of $I_{A_G}$ are completely characterized in terms of graphs
(Proposition \ref{circuitwalk}).
It is known that $K[A_G]$ is normal if and only if $G$ satisfies 
``the odd cycle condition" (Proposition \ref{oddcyclecondition}).
In \cite[Section 3]{indispensable}, 
generators of $I_{A_G}$ are studied when $K[A_G]$ is normal.
It is essentially shown in \cite[Proof of Lemma 3.2]{indispensable}
that, if $K[A_G]$ is normal, then 
we have $I_{A_G}= \langle \squarefreeG  \rangle$.
Martinez-Bernal--Villarreal \cite[Theorem 3.2]{V}
also proved this fact and claimed that 
the converse is true.
However, as they stated in \cite[Note added in proof]{V},
the converse is false in general.
Several classes of counterexamples are given in Section \ref{graphsection}.

The content of this paper is as follows.
In Section 1,
we study toric ideals having squarefree quadratic initial ideals.
For such configurations,
a sufficient condition to be generated by circuits is given.
In particular, squarefree Veronese subrings,
the second Veronese subrings and configurations
arising from root systems satisfy the condition.
In Section 2, we study toric ideals of finite graphs.
We characterize the graphs $G$ whose toric ideals
are generated by $\squarefreeG$.
A similar result is given for $\squarefreesquarefreeG$.
By this characterization,
we construct classes of graphs $G$
such that
$K[A_G]$ is nonnormal and that
$I_{A_G} = \langle \squarefreeG \rangle
= \langle \squarefreesquarefreeG \rangle$.

\section{Configurations with squarefree quadratic initial ideals}

In this section, we study several classes of
toric ideals with squarefree quadratic initial ideals.
It is known 
\cite[Proposition 13.15]{Stu}
that, if a toric ideal $I_A$ has a squarefree initial ideal,
then $K[A]$ is normal.
First, we show a fundamental fact on quadratic binomials
in toric ideals.
(Since we assume
the columns of $A$ are pairwise distinct,
$I_A$ has no binomials of degree 1.)

\begin{Proposition}
\label{zeroone}
Let $A = (a_{ij})\in \ZZ^{d \times n}$ be a configuration.
Suppose that, for each $1 \leq i \leq d$, there exists $z_i \in \ZZ$ such that
$z_i -1 \leq a_{ij} \leq z_i+1$ for all $1 \leq j \leq n$.
Then, any quadratic binomial in $I_A$
belongs to $\squarefree$.
Moreover, if, for each $1 \leq i \leq d$, there exists $z_i \in \ZZ$ such that
$z_i \leq a_{ij} \leq z_i+1$ for all $1 \leq j \leq n$,
then
any quadratic binomial in $I_A$
belongs to $\squarefreesquarefree$.
\end{Proposition}

\begin{proof}
Suppose that, for each $1 \leq i \leq d$, there exists $z_i \in \ZZ$ such that
$z_i -1 \leq a_{ij} \leq z_i+1$ for all $1 \leq j \leq n$.
It is known \cite[Lemma 4.14]{Stu} that 
there exists a vector ${\bf w} \in \RR^d$ such that
${\bf w} \cdot A = (1,1,\ldots,1)$.
Hence, by elementary row operations,
we may assume that $A$ is a $(0,\pm 1)$-configuration.
Let $A =({\bf a}_1,\ldots, {\bf a}_n) \in \ZZ^{d \times n}$
and
let $f \in I_A$ be a quadratic binomial.
Since the columns of $A$ are pairwise distinct,
$f$ is of the form
either $x_1 x_2 - x_3 x_4$ or $x_1 x_2 -x_3^2$.
Note that $\sharp|{\rm supp} (h)| \geq 3$ for any binomial $h \in I_A$.
Hence $f$ is a circuit if $f = x_1 x_2 -x_3^2$.

Let $f = x_1 x_2 - x_3 x_4$ and suppose that $f \notin C_A$.
By \cite[Lemma 4.10]{Stu},
there exists a circuit $g = X^{\bf u} - X^{\bf v} \in C_A$ such that
${\rm supp}(X^{\bf u} ) \subset\{x_1,x_2\}$
and
${\rm supp}(X^{\bf v} ) \subset\{x_3,x_4\}$.
Since $f$ is not a circuit,
$\sharp|{\rm supp} (g)| < 4$.
Hence we have $\sharp|{\rm supp}(g)|=3$.
Thus we may assume that
$g = x_1^a x_2^b - x_3^c$
where $1 \leq a,b,c \in \ZZ$.
Then, $a \cdot {\bf a}_1 +b \cdot {\bf a}_2 = c \cdot {\bf a}_3$
and $a + b =c$.
Let ${\bf a}_k = (a_1^{(k)}, a_2^{(k)},\ldots,a_d^{(k)})^\top $ for $k =1,2,3$.
Since ${\bf a}_1,{\bf a}_2, {\bf a}_3$ are $(0,\pm 1)$-vectors,
we have the following for each $1 \leq j \leq d$:
\begin{itemize}
\item
If $a_j^{(3)} = 1$, then $a_j^{(1)} = a_j^{(2)} = 1$.
\item
If $a_j^{(3)} = -1$, then $a_j^{(1)} = a_j^{(2)} = -1$.
\end{itemize}
Since ${\bf a}_1$ and ${\bf a}_3$ are distinct,
there exists $1 \leq k \leq d$ such that
$a_k^{(3)} = 0$ and $a_k^{(1)}  \neq 0$.
Then $a \cdot a_k^{(1)}  + b \cdot a_k^{(2)} = 0$ and
hence $a = b $ and $a_k^{(2)}  = - a_k^{(1)} $.
Note that $g = x_1^a x_2^a - x_3^{2a}$ should be irreducible.
It then follows that $a =1$ and $g =  x_1 x_2 - x_3^2$.
Thus $f - g = x_3^2 - x_3 x_4$ belongs to $I_A$ and hence
${\bf a}_3 = {\bf a}_4$, a contradiction.
Therefore $f \in C_A$.

Suppose that, for each $1 \leq i \leq d$, there exists $z_i \in \ZZ$ such that
$z_i \leq a_{ij} \leq z_i+1$ for all $1 \leq j \leq n$.
By elementary row operations,
we may assume that $A$ is a $(0,1)$-configuration.
Let $f = x_1 x_2 -x_3^2 \in I_A$.
Then, ${\bf a}_1 + {\bf a}_2 = 2 \cdot {\bf a}_3$.
Since ${\bf a}_1,{\bf a}_2,{\bf a}_3$ are (0,1)-vectors,
it follows that ${\bf a}_1 = {\bf a}_2= {\bf a}_3$, a contradiction.
\end{proof}

By Proposition \ref{zeroone},
we can prove that several classes
of toric ideals are generated by circuits.

\subsection{Veronese and squarefree Veronese configurations}

Let $2 \leq d, r \in \ZZ$ and $V_d^{(r)} = ( {\bf a}_1,\ldots, {\bf a}_n ) \in \ZZ^{d \times n}$
be the matrix where
$$\{{\bf a}_1,\ldots, {\bf a}_n\}
=
\left\{
(\alpha_1,\ldots, \alpha_d)^\top \in \ZZ^d
\ \left| \ 
\alpha_i \geq 0 , \ \sum_{i=1}^d \alpha_i = r
\right.
\right\}.$$
Then, $K[V_d^{(r)}]$ is called the {\em $r$-th Veronese subring} of $K[t_1,\ldots,t_d]$.
On the other hand, let $SV_d^{(r)} = ( {\bf a}_1,\ldots, {\bf a}_n ) \in \ZZ^{d \times n}$
be the matrix where
$$\{{\bf a}_1,\ldots, {\bf a}_n\}
=
\left\{
(\alpha_1,\ldots, \alpha_d)^\top \in \{0,1\}^d
\ \left| \ 
\sum_{i=1}^d \alpha_i = r
\right.
\right\}.$$
Then, $K[SV_d^{(r)}]$ is called the {\em $r$-th squarefree Veronese subring} of $K[t_1,\ldots,t_d]$.
It is known (e.g., \cite[Chapter 14]{Stu}) that

\begin{Proposition}
\label{Veronese}
Toric ideals $I_{ V_d^{(r)} }$ and $I_{ SV_d^{(r)} }$ have squarefree quadratic initial
ideals and hence $K[ V_d^{(r)} ]$ and $K[ SV_d^{(r)} ]$ are normal.
\end{Proposition}

We characterize such toric ideals that are generated by circuits.

\begin{Theorem}
Let $2 \leq d, r \in \ZZ$.
Then, we have the following:
\begin{itemize}
\item[(i)]
For $A= SV_d^{(r)}$, 
the toric ideal $I_A$ is generated by $\squarefreesquarefree$.
\item[(ii)]
For $A= V_d^{(r)}$,
the toric ideal $I_A$ is generated by $\squarefree$
if and only if $r = 2$.
\end{itemize}
\end{Theorem}

\begin{proof}
First, by Propositions \ref{zeroone} and \ref{Veronese},
(i) and the ``if" part of (ii) holds.

Let $r \geq 3$.
Since $K[V_2^{(r)}]$ is a combinatorial pure subring (see \cite{cpure} for details) of $K[V_d^{(r)}]$
for all $d > 2$, it is sufficient to show that
$I_{ V_2^{(r)} }$ is not generated by circuits.
Recall that the configuration $V_2^{(r)}$ is 
$
\begin{pmatrix}
r & r-1 & r-2 & r-3 & \cdots & 0\\
0 & 1  & 2 & 3 & \cdots & r
\end{pmatrix}.
$
Then the binomial $x_1 x_4 - x_2 x_3 \in I_{ V_2^{(r)} }$ is not a circuit 
since $x_2^2- x_1 x_3$ belongs to $ I_{ V_2^{(r)} }$.
Suppose that 
$0 \neq x_1 x_4 - x_i x_j$ belongs to $I_{ V_2^{(r)} }$.
Then ${\bf a}_i + {\bf a}_j = {\bf a}_1 + {\bf a}_4 = (2 r -3, 3)^\top $.
Since the last coordinate of ${\bf a}_i + {\bf a}_j$ is 3,
it follows that $\{i,j\}$ is either $\{1,4\}$ or $\{2,3\}$.
Hence $x_1 x_4 - x_i x_j = x_1 x_4 - x_2 x_3$.
Thus $x_1 x_4 - x_2 x_3 $ is not generated by other binomials in $ I_{ V_2^{(r)} }$
as desired.
\end{proof}

\subsection{Configurations arising from root systems}

For an integer $d \geq 2$, 
let $\Phi \subset \ZZ^d$ be one of the classical irreducible root systems
${\bf A}_{d - 1}$, ${\bf B}_d$, ${\bf C}_d$ and
${\bf D}_d$ (\cite[pp. 64 -- 65]{rootbook}) and write $\Phi^{(+)}$
for the set consisting of the origin
of $\RR^d$ together with all positive roots of $\Phi$.
More explicitly,
\begin{eqnarray*}
{\bf A}_{d-1}^{(+)}
&=& \{ {\bf 0} \} 
\cup 
\{ {\bf e}_i - {\bf e}_j  \ | \ 1 \leq  i < j \leq d\}\\
{\bf B}_d^{(+)}
&=& 
{\bf A}_{d-1}^{(+)}
\cup
\{ {\bf e}_1,\ldots, {\bf e}_d\}
\cup
\{ {\bf e}_i + {\bf e}_j  \ | \  1 \leq i < j \leq d\}\\
{\bf C}_d^{(+)}
&=& 
{\bf A}_{d-1}^{(+)}
\cup
\{ {\bf e}_i + {\bf e}_j  \ | \  1 \leq i \leq j \leq d\}\\
{\bf D}_d^{(+)}
&=& 
{\bf A}_{d-1}^{(+)}
\cup
\{ {\bf e}_i + {\bf e}_j  \ | \  1 \leq i < j \leq d\}
\end{eqnarray*}
where ${\bf e}_i$ is the $i$th unit coordinate vector of $\RR^d$ and 
${\bf 0}$ is the origin of $\RR^d$.
For each $\Phi^{(+)} \in 
\{{\bf A}_{d-1}^{(+)}, {\bf B}_d^{(+)} , {\bf C}_d^{(+)}, {\bf D}_d^{(+)}\}$, 
we identify $\Phi^{(+)}$ with the matrix whose columns are $\Phi^{(+)}$
and associate the configuration
$$
\widetilde{\Phi}^{(+)} = 
\left(
\begin{array}{ccc}
\\
 & \Phi^{(+)} & \\
\\
\hline
1 & \cdots & 1
\end{array}
\right).
$$

\begin{Proposition}[\cite{GGP}, \cite{quadroot}]
\label{rootnormal}
Working with the same notation as above.
Then the toric ideal $I_{\widetilde{\Phi}^{(+)}}$ has a squarefree quadratic 
initial ideal
and hence $K[\widetilde{\Phi}^{(+)}]$ is normal.
\end{Proposition}

By Proposition \ref{zeroone}, we have the following.

\begin{Corollary}
If $A \in 
\left\{
\widetilde{\bf A}_{d-1}^{(+)}, \widetilde{\bf B}_d^{(+)} ,
\widetilde{\bf C}_d^{(+)} , \widetilde{\bf D}_d^{(+)}
\right\}$, then
$I_A$ is generated by quadratic binomials in $\squarefree$.
\end{Corollary}

\begin{proof}
Since $\widetilde{\bf A}_{d-1}^{(+)}$, $\widetilde{\bf B}_d^{(+)}$ and
 $\widetilde{\bf D}_d^{(+)}$
are $(0,\pm 1)$ configurations, 
by Propositions \ref{zeroone} and \ref{rootnormal},
$I_A$ is generated by quadratic binomials in $\squarefree$
if $A \in 
\left\{
\widetilde{\bf A}_{d-1}^{(+)}, \widetilde{\bf B}_d^{(+)} ,
 \widetilde{\bf D}_d^{(+)}
\right\}$.

Let $A =  \widetilde{\bf C}_d^{(+)}$.
By elementary row operations,
one can transform
the matrix $A$
as follows:
$$
A
\longrightarrow 
\left(
\begin{array}{ccc|ccc}
 &  & & &  & \\
 & {\bf A}_{d-1}^{(+)} & & & P & \\
 &  & & &  & \\
\hline
1 & \cdots & 1 & 0 & \cdots & 0
\end{array}
\right)
\longrightarrow 
\left(
\begin{array}{ccc|ccc}
 &  & & &  & \\
 & {\bf A}_{d-1}^{(+)} + {\bf 1} & & & P & \\
 &  & & &  & \\
\hline
1 & \cdots & 1 & 0 & \cdots & 0
\end{array}
\right) =Q
$$
where ${\bf 1}$ is the matrix with all entries equal to one
and
$P$ is the matrix whose columns are 
$\{ {\bf e}_i + {\bf e}_j  \ | \  1 \leq i \leq j \leq d\}$.
Since $Q$ is a $(0,1,2)$-configuration,
$I_Q = I_A$
is generated by quadratic binomials in $\squarefree$
by Propositions \ref{zeroone} and \ref{rootnormal}.
\end{proof}

\section{Configurations arising from graphs}

\label{graphsection}

In this section, we study toric ideals arising from graphs.
First, we introduce some graph terminology.
A {\em walk} of $G$ of length $q$ is a sequence 
$\Gamma=(e_{i_1}, e_{i_2}, \ldots, e_{i_q})$
of edges of $G$, 
where 
$e_{i_k} = \{ u_k, v_k \}$
for $k = 1, \ldots, q$, such that
$v_k = u_{k+1}$ for $k = 1, \ldots, q - 1$. 
Then,
\begin{itemize}
\item
A walk $\Gamma$ is called a {\em path}
if $\sharp |\{u_1,\ldots,u_q,v_q\}|=q+1$.
\item
A walk $\Gamma$ is called a {\em closed walk}
if $v_{q} = u_{1}$.
\item
A walk $\Gamma$ is called a {\em cycle}
if $v_{q} = u_{1}$,
$q \geq 3$ and
$\sharp |\{u_1,\ldots,u_q\}|=q$.
\end{itemize}
For a cycle $\Gamma=(e_{i_1}, e_{i_2}, \ldots, e_{i_q})$,
an edge $e=\{ s,t  \}$ of $G$
is called a {\em chord} of $\Gamma$ if 
$s$ and $t$ are vertices of $\Gamma$ and 
if $e \notin \{e_{i_1}, e_{i_2}, \ldots, e_{i_q}\}   $.
A cycle $\Gamma$ is called {\em minimal} if 
$\Gamma$ has no chord.
Let $\Gamma = (e_{i_1}, e_{i_2}, \ldots, e_{i_{2 q}})$
where 
$e_{i_k} = \{ u_k, v_k \}$
for $k = 1, \ldots, 2 q$ be an even closed walk of $G$.
Then it is easy to see that the binomial
$$f_{\Gamma} = \prod_{\ell=1}^q x_{u_{2 \ell -1} v_{2 \ell -1}} -  \prod_{\ell=1}^q x_{u_{2 \ell} v_{2 \ell}}$$
belongs to $I_{A_G}$.
Circuits of $I_{A_G}$ are characterized
in terms of graphs (see, e.g., \cite[Lemma 9.8]{Stu}).

\begin{Proposition}
\label{circuitwalk}
Let $G$ be a finite connected graph.
Then $f \in C_{A_G}$ if and only if 
$f = f_\Gamma$ for some even closed walk $\Gamma$
which is one of the following even closed walk:
\begin{itemize}
\item[(i)]
$\Gamma$ is an even cycle of $G$;
\item[(ii)]
$\Gamma = (C_1,C_2)$, where $C_1$ and $C_2$ are odd cycles of $G$
having exactly one common vertex;
\item[(iii)]
$\Gamma = (C_1,e_1,\ldots,e_r ,C_2,e_r,\ldots, e_1)$,
where $C_1$ and $C_2$ are odd cycles of $G$
having no common vertex and where $(e_1,\ldots,e_r)$
is a path of $G$ which combines a vertex of $C_1$
and a vertex of $C_2$.
\end{itemize}
In particular, $f \notin \squarefreeG$ if and only if 
$\Gamma$ satisfies {\rm (iii)} and $r >1$.
\end{Proposition}

Moreover,
it is known \cite[Lemma 3.2]{quadgene} that 

\begin{Proposition}
\label{primitive}
Let $G$ be a finite connected graph.
Then $I_{A_G}$ is generated by all $f_\Gamma$
where $\Gamma$ is one of the following even closed walk:
\begin{itemize}
\item[(i)]
$\Gamma$ is an even cycle of $G$;
\item[(ii)]
$\Gamma = (C_1,C_2)$, where $C_1$ and $C_2$ are odd cycles of $G$
having exactly one common vertex;
\item[(iii)]
$\Gamma = (C_1,\Gamma_1,C_2,\Gamma_2)$,
where $C_1$ and $C_2$ are odd cycles of $G$
having no common vertex and where $\Gamma_1$ and $\Gamma_2$
are walks of $G$ both of which combine a vertex $v_1$ of $C_1$
and a vertex $v_2$ of $C_2$.
\end{itemize}
\end{Proposition}

See also \cite{RTT} for characterization
of generators of $I_{A_G}$.
The normality of $K[A_G]$ is characterized
in terms of graphs.

\begin{Proposition}[\cite{normalOH}]
\label{oddcyclecondition}
Let $G$ be a finite connected graph.
Then $K[A_G]$ is normal if and only if $G$ satisfies the odd cycle condition,
i.e., for an arbitrary two odd cycles $C_1$ and $C_2$ in $G$ without common vertex,
there exists an edge of $G$ joining a vertex of $C_1$ with a vertex of $C_2$.
\end{Proposition}

Let $A =  ({\bf a}_1, \ldots, {\bf a}_n) \in \ZZ^{d \times n}$
be a configuration.
Given binomial $f = u-v \in I_A$,
we write $T_f$ for the set of those
variables $t_i$ such that $t_i$ divides
$\pi(u) (= \pi(v) )$.
Let $K[T_f ] = K[\{t_i \ | \ t_i \in T_f \}]$ and
let
$A_f$ be the matrix whose columns are
$\{ {\bf a}_i \ | \ T^{{\bf a}_i} \in K[T_f]\}$.
The toric ideal $I_{A_f}$ coincides with
$I_A \cap K[\{x_i \ | \ \pi(x_i) \in K[T_f ]\}]$.
A binomial $f \in I_A$ is called {\em fundamental} if
$I_{A_f}$ is generated by $f$.
A binomial $f \in I_A$ is called
{\em indispensable} if, 
for any system of binomial
generators $F$ of $I_A$, either $f$ or $-f$ belongs to $F$.
A binomial $f \in I_A$ is called
{\em not redundant} if $f$ belongs to a minimal system of 
binomial generators of $I_A$.
Given binomial $ f  \in I_A$, 
it is known \cite{OHtable} that 
\begin{itemize}
\item
$f$ is fundamental $ \Longrightarrow $ $f$ is a circuit
\item
$f$ is fundamental $ \Longrightarrow $
$f$ is indispensable $ \Longrightarrow $
$f$ is not redundant
\end{itemize}
hold in general.

We give a characterization of toric ideals of graphs
generated by $\squarefreeG$. 

\begin{Theorem}
\label{graphcircuit}
Let $G$ be a finite connected graph.
Then the following conditions are equivalent:
\begin{itemize}
\item[(i)]
$I_{A_G} = \langle \squarefreeG \rangle$;
\item[(ii)]
Any circuit in $C_{A_G} \setminus \squarefreeG$ is
redundant;
\smallskip
\item[(iii)]
Any circuit in $C_{A_G} \setminus \squarefreeG$ is
not indispensable;
\smallskip
\item[(iv)]
Any circuit in $C_{A_G} \setminus \squarefreeG$ is
not fundamental;
\item[(v)]
There exists no induced subgraph of $G$
consisting of two odd cycles $C_1$, $C_2$ 
having no common vertex and 
a path of length $\geq 2$
which connects a vertex of $C_1$ and 
a vertex of $C_2$.
\end{itemize}
In particular, if $G$ satisfies the odd cycle condition,
then $G$ satisfies 
{\rm (v)}.
\end{Theorem}

In order to prove Theorem \ref{graphcircuit},
we need the following lemma:

\begin{Lemma}
\label{keylemma}
Let $G$ be a finite connected graph which satisfies
the condition {\rm (v)} in Theorem \ref{graphcircuit}.
Let $C$ and $C'$ be odd cycles of $G$
having no common vertex
and let $\Gamma$ a path of $G$
which combines a vertex $v$ of $C$
and a vertex $v'$ of $C'$.
Then, at least one of the following holds:
\begin{itemize}
\item[(a)]
There exists an edge of $G$ joining a vertex of $C$ with a vertex of $C'$.
\item[(b)]
There exists an edge of $G$ joining a vertex $p$ of $C$ with a vertex $q$ of $\Gamma$
where $q \neq v$ and $\{p,q\} \notin \Gamma$.
\item[(c)]
There exists an edge of $G$ joining a vertex $p$ of $C'$ with a vertex $q$ of $\Gamma$
where $q \neq v'$ and $\{p,q\} \notin \Gamma$.
\end{itemize}
\end{Lemma}

\begin{proof}
The proof is by induction on the sum of the length of $C$ and $C'$.

(Step 1.)
Suppose that $C$ and $C'$ are cycle of length $3$.
Then, $C$ and $C'$ are minimal.
If $C$, $C'$ and $\Gamma$ satisfy none of (a), (b) and (c),
then, by the condition (v),
it follows that $\Gamma$ is not an induced subgraph of $G$.
Then there exists a path $\Gamma'$ which combines $v$ and $v'$
whose vertex set is a proper subset of the vertex set of $\Gamma$.
By repeating the same argument, we may assume that 
the path $\Gamma'$ is an induced subgraph of $G$.
This contradicts to the condition (v).

(Step 2.)
Let $C$ and $C'$ be odd cycles of $G$
having no common vertex
and let $\Gamma$ a path of $G$
which combines a vertex $v$ of $C$
and a vertex $v'$ of $C'$.
If both $C$ and $C'$ are minimal,
then one of (a), (b) and (c) follows from the same argument
in Step 1.
Suppose that $C$ is not minimal, i.e.,
there exists
 a chord $e$ of $C$.
It is easy to see that there exists a unique odd cycle $C_e$ such that
$e \in E(C_e) \subset E(C) \cup \{ e\}$.
Note that the length of $C_e$ is less than that of $C$.

If $v$ is a vertex of $C_e$, then 
$C_e$, $C'$ and $\Gamma$ satisfy 
one of (a), (b) and (c) by the hypothesis of induction.
Thus $C$, $C'$ and $\Gamma$ satisfy the same condition.

Suppose that $v$ is not a vertex of $C_e$
for any chord $e$ of $C$.
Then $C_e$, $C'$ and
a path
$\Gamma'=(e_{i_1},\ldots,e_{i_s},\Gamma)$
where $(e_{i_1},\ldots,e_{i_s})$ is a part of $C$ satisfy 
one of (a), (b) and (c) by the hypothesis of induction.
We may assume that $s$ $(\geq 1)$ is minimal.
If $C$, $C'$ and $\Gamma$ satisfy none of (a), (b) and (c),
then $C_e$, $C$ and $\Gamma'$ satisfy the condition (b)
where $q$ is not a vertex of $\Gamma$.
This contradicts the minimality of $s$.
\end{proof}

\begin{proof}[Proof of Theorem \ref{graphcircuit}]
In general,
(i) $\Longrightarrow $
(iii)
and
(ii) $\Longrightarrow $
(iii) $\Longrightarrow $
(iv)
hold.
Moreover, by Proposition \ref{circuitwalk},
(iv) $\Longrightarrow $ (v) holds.

\smallskip

\noindent
\underline{\bf (v) $\Longrightarrow $ (i)}
Suppose that $G$ satisfies the condition (v).
Let $f=f_\Gamma \notin \squarefreeG$ where $\Gamma$ 
is an even closed walk satisfying the condition (iii) in
 Proposition \ref{primitive},
i.e.,
$\Gamma = (C_1,\Gamma_1,C_2,\Gamma_2)$,
where $C_1$ and $C_2$ are odd cycles of $G$
having no common vertex,
 and $\Gamma_1$ and $\Gamma_2$
are walks of $G$ both of which combine a vertex $v_1$ of $C_1$
and a vertex $v_2$ of $C_2$.
By Propositions \ref{circuitwalk} and \ref{primitive},
it is sufficient to show that $f$ is redundant.
Since $f$ does not belong to $\squarefreeG$,
at least one of $\Gamma_i$ is of length $>1$.
We may assume that, except for starting and ending vertices,
each $\Gamma_i$ does not contain the vertices of two odd cycles.
(Otherwise, $\Gamma$ separates into two even closed walk
and hence $f$ is redundant.)

If there exists an edge of $G$ joining a vertex of $C_1$ with a vertex of $C_2$,
then $f$ is redundant by \cite[Proof of Lemma 3.2]{indispensable}.
Suppose that there exists no such an edge.
(Then, in particular, the length of $\Gamma_i$ is greater than 1 for $i=1,2$.)
By Lemma \ref{keylemma},
 there exists an edge of $G$ joining a vertex $p$ of $C_1$
with a vertex $q \ (\neq v_1)$ of $\Gamma_1$
and $\{p,q\}$ does not belong to $\Gamma$.
Let $C_1= ( V_1 ,  V_2 )$ and $\Gamma_1 = (W_1, W_2)$
where
\begin{itemize}
\item
$V_1$ and $V_2$ are paths joining $v_1$ and $p$;
\item
$W_i$ is a walk joining $v_i$ and $q$ for $i = 1,2$.
\end{itemize}
Since the length of $C_1$ is odd,
we may assume that the length of the walk $(V_1,W_1)$ is odd.
Note that both
$\Gamma_3 = (V_1,W_1,\{ q, p \})$
and
$\Gamma_4 = (V_2,\Gamma_2,C_2, W_2,\{ q, p \})$
are even closed walks.
It then follows that $f \in \langle f_{\Gamma_3}, f_{\Gamma_4}  \rangle$
and $\deg (f_{\Gamma_3} ) ,\deg (f_{\Gamma_4} ) < \deg (f) $.
Hence $f$ is redundant.

Thus, $f$ is redundant and hence $G$ satisfies the condition (i).

\smallskip

\noindent
\underline{\bf (v) $\Longrightarrow $ (ii)}
Suppose that $G$ satisfies the condition (v).
Let $f=f_\Gamma \in C_{A_G} \setminus 
\squarefreeG$ where $\Gamma = (C_1,e_1,\ldots,e_r,C_2,e_r,\ldots,e_1)$
$(r>1) $
is an even closed walk satisfying the condition (iii) in Proposition \ref{circuitwalk}.
Then $f$ is redundant by the same argument above.
Thus $G$ satisfies the condition (ii).
\end{proof}

A similar theorem holds for $\squarefreesquarefreeG$.

\begin{Theorem}
\label{graphcircuit2}
Let $G$ be a finite connected graph.
Then the following conditions are equivalent:
\begin{itemize}
\item[(i)]
$I_{A_G} = \langle \squarefreesquarefreeG \rangle$;
\item[(ii)]
Any circuit in $C_{A_G} \setminus \squarefreesquarefreeG$ is
redundant;
\smallskip
\item[(iii)]
Any circuit in $C_{A_G} \setminus \squarefreesquarefreeG$ is
not indispensable;
\smallskip
\item[(iv)]
Any circuit in $C_{A_G} \setminus \squarefreesquarefreeG$ is
not fundamental;
\item[(v)]
There exists no induced subgraph of $G$
consisting of two odd cycles $C_1$, $C_2$ 
having no common vertex and 
a path of length $\geq 1$
which connects a vertex of $C_1$ and 
a vertex of $C_2$.
\end{itemize}
\end{Theorem}

\begin{proof}
As stated in Proof of Theorem \ref{graphcircuit},
it is sufficient to show
``(v) $\Longrightarrow $ (i)"
and
``(v) $\Longrightarrow $ (ii)."
Suppose that $G$ satisfies (v).
By Theorem \ref{graphcircuit}, 
$I_{A_G} = \langle \squarefreeG \rangle$ and 
any circuit in $C_{A_G} \setminus \squarefreeG$ is
redundant.
Thus, in order to prove (i) and (ii), 
it is sufficient to show that 
any circuit in $ \squarefreeG  \setminus \squarefreesquarefreeG$ is
redundant.
Let $f$ be a binomial in 
$ \squarefreeG  \setminus \squarefreesquarefreeG$.
By Proposition \ref{circuitwalk},
$f = f_\Gamma$ where $\Gamma$ is an even closed walk
which consists of two odd cycles $C_1$ and $C_2$
having no common vertex and an edge $e_0$ of $G$
which combines a vertex of $C_1$ and a vertex of $C_2$.
Since $G$ satisfies the condition (v),
$\Gamma$ is not an induced subgraph of $G$.
If there exists an edge $e' (\neq e_0)$ of $G$
joining a vertex of $C_1$ with a vertex of $C_2$,
then $f$ is redundant by \cite[Proof of Lemma 3.2]{indispensable}.
Suppose that there exists no such an edge.
Since $\Gamma$ is not an induced subgraph of $G$,
we may assume that $C_1$ is not minimal.
Then there exists a chord $e$ of $C_1$
and
an odd cycle $C_e$ such that
$e \in E(C_e) \subset E(C) \cup \{ e\}$.
If $v$ is a vertex of $C_e$, then $f$ is redundant
by \cite[Proof of Lemma 3.2]{indispensable}.
Suppose that $v$ is not a vertex of $C_e$
for any chord $e$ of $C$.
Note that $C_e$, $C_2$ and $\Gamma =(e_{i_1},\ldots,e_{i_s},e_0)$
where $(e_{i_1},\ldots,e_{i_s})$ is a part of $C$ satisfy 
one of (a), (b) and (c) in Lemma \ref{keylemma}.
Suppose that $s$ is minimal.
Since 
there exists no edge $e' (\neq e_0)$ of $G$
joining a vertex of $C_1$ with a vertex of $C_2$,
 $C_e$, $C_2$ and $\Gamma $
 satisfy the condition (b).
This contradicts the minimality of $s$.
\end{proof}

Using Theorems \ref{graphcircuit} and \ref{graphcircuit2},
we give several classes of graphs $G$ such that
$I_{A_G} = \langle \squarefreeG \rangle$
and $K[A_G]$ is nonnormal.

\begin{Example}
{\em
Let $G$ be the graph whose vertex-edge incidence matrix is
$$
A_G=
\begin{pmatrix}
1 & 0 & 0 & 1 & 0 & 0 & 0  & 1 & 0\\
1 & 1 & 0 & 0 & 0 & 0 & 0  & 0 & 0\\
0 & 1 & 1 & 0 & 0 & 0 & 0  & 1 & 0\\
0 & 0 & 1 & 1 & 1 & 0 & 0  & 0 & 0\\
0 & 0 & 0 & 0 & 1 & 1 & 0  & 0 & 1\\
0 & 0 & 0 & 0 & 0 & 1 & 1  & 0 & 0\\
0 & 0 & 0 & 0 & 0 & 0 & 1  & 0 & 1
\end{pmatrix}
$$
Then, 
$I_{A_G}$ is generated by the circuits
$x_1 x_3 - x_2 x_4, \ \  x_3 x_4 x_6 x_9  -x_5^2 x_7 x_8$
(\cite[Example 3.5]{indispensable}).
Since $G$ does not satisfy the odd cycle condition,
$K[A_G]$ is not normal.
}
\end{Example}

\begin{Example}
\label{sqfexample}
{\em
Let $G$ be the graph whose vertex-edge incidence matrix is
$$
A_G=
\begin{pmatrix}
1 & 0 & 0 & 1 & 0 & 0 & 0 & 0 & 1 & 0\\
1 & 1 & 0 & 0 & 0 & 0 & 0 & 0 & 0 & 0\\
0 & 1 & 1 & 0 & 0 & 0 & 0 & 0 & 1 & 0\\
0 & 0 & 1 & 1 & 1 & 0 & 0 & 1 & 0 & 0\\
0 & 0 & 0 & 0 & 1 & 1 & 0 & 0 & 0 & 1\\
0 & 0 & 0 & 0 & 0 & 1 & 1 & 0 & 0 & 0\\
0 & 0 & 0 & 0 & 0 & 0 & 1 & 1 & 0 & 1
\end{pmatrix}.
$$
Then,
$I_{A_G}$ is generated by the circuits
$x_5 x_7 - x_6 x_8 , \ \  x_1 x_3 - x_2 x_4, \ \  x_3 x_4 x_{10} -x_5 x_8 x_9.$
Since $G$ does not satisfy the odd cycle condition,
$K[A_G]$ is not normal.
}
\end{Example}

Example \ref{sqfexample} is 
the most simple nonnormal example
whose toric ideal is generated by
circuits $u - v$ such that
the two monomials $u$ and $v$ are squarefree.
In fact,

\begin{Proposition}
If $I_A$ is generated by binomials
$f_1= X^{{\bf u}^+} - X^{{\bf u}^-}$,
$f_2=X^{{\bf v}^+} - X^{{\bf v}^-}$
such that 
$X^{{\bf u}^+}$, $X^{{\bf u}^-}$,
$X^{{\bf v}^+}$ and $X^{{\bf v}^-}$ are squarefree,
then there exists a monomial order such that 
$\{f_1,f_2\}$ is a Gr\"obner basis of $I_A$
and hence $K[A]$ is normal.
\end{Proposition}

\begin{proof}
Suppose that
$x_i \in {\rm supp} (X^{{\bf u}^+})  \cap {\rm supp} (X^{{\bf v}^-})$ and 
$x_j \in {\rm supp} (X^{{\bf u}^-})  \cap {\rm supp} (X^{{\bf v}^+})$.
Let ${\bf w} = {\bf u} + {\bf v} \in {\rm Ker}_\ZZ (A)$
and 
$g=X^{{\bf w}^+} - X^{{\bf w}^-}$.
Then $g$ belongs to $I_A$.
Since $x_i$ belongs to 
${\rm supp} (X^{{\bf u}^+})  \cap {\rm supp} (X^{{\bf v}^-})$,
${\rm supp} (g)$ does not contain $x_i$.
Similarly, since $x_j$ belongs to 
${\rm supp} (X^{{\bf u}^-})  \cap {\rm supp} (X^{{\bf v}^+})$,
${\rm supp} (g)$ does not contain $x_j$.
Hence $g$ is not generated by $f_1$ and $f_2$.
This contradicts that $g \in I_A$.
Thus, we may assume that 
${\rm supp} (X^{{\bf u}^+})  \cap {\rm supp} (X^{{\bf v}^-}) = \emptyset$
and
${\rm supp} (X^{{\bf u}^+})  \cap {\rm supp} (X^{{\bf v}^+}) = \emptyset
$.
Let $<$ be a lexicographic order induced by the ordering
$$
{\rm supp}(X^{{\bf u}^+}) > 
{\rm supp}(X^{{\bf v}^+} )> 
\mbox{ other variables} .$$
Then ${\rm in}_<(f_1) = X^{{\bf u}^+}$ and
${\rm in}_<(f_2) = X^{{\bf v}^+}$ are relatively prime.
Hence $\{f_1,f_2\}$ is a Gr\"obner basis of $I_A$.
Since both ${\rm in}_<(f_1)$ and
${\rm in}_<(f_2) $ are squarefree, $K[A]$ is normal.
\end{proof}

Let $G_1 = (V_1, E_1)$ and 
$G_2 = (V_2, E_2)$ be graphs such that $V_1 \cap V_2$ is a clique of both graphs.
The new graph $G = G_1 \sharp G_2$ with the vertex set $V = V_1 \cup V_2$
and edge set $E=E_1 \cup E_2$ is called the {\em clique sum} of $G_1$ and $G_2$
along $V_1 \cap V_2$.
If the cardinality of $V_1 \cap V_2$ is $k+1$, this operation is called a {\em $k$-sum} of the graphs.

\begin{Example}
{\em
Let $G$ be the 0-sum of two complete graphs having at least 4 vertices.
Then, $G$ satisfies the condition (v) in Theorem \ref{graphcircuit2}
and hence $I_{A_G}$ is
generated by $\squarefreesquarefreeG$.
Since $G$ does not satisfy the odd cycle condition,
$K[A_G]$ is not normal.
On the other hand, 
by the criterion \cite[Theorem 2.1]{R1},
it follows that $K[A_G]$ does not satisfy Serre's condition ($R_1$).
}
\end{Example}

\begin{Example}
{\em
Let $G$ be the 1-sum of two complete graphs having at least 5 vertices.
Then, $G$ satisfies the condition (v) in Theorem \ref{graphcircuit2}
and hence $I_{A_G}$ is
generated by $\squarefreesquarefreeG$.
Since $G$ does not satisfy the odd cycle condition,
$K[A_G]$ is not normal.
On the other hand, 
by the criterion \cite[Theorem 2.1]{R1},
it follows that $K[A_G]$ satisfies Serre's condition ($R_1$).
}
\end{Example}

\end{document}